\documentclass[11pt,a4paper,reqno]{amsart}
\usepackage[utf8]{inputenc}
\usepackage[english]{babel}
\usepackage{mathtools}
\usepackage{amsmath,amsthm,amssymb,amsfonts}
\numberwithin{equation}{section}

\usepackage{enumitem}

\usepackage{graphicx}
\usepackage[pagebackref=true]{hyperref}
\usepackage[left=2.5cm,right=2.5cm,top=2.5cm,bottom=2.5cm]{geometry}

\renewcommand*{\backref}[1]{}
\renewcommand*{\backrefalt}[4]{{\tiny%
    (\ifcase #1 Not cited.%
          \or Cited on p.~#2.%
          \else Cited on pp. #2.%
    \fi%
    )}}

\newcommand{\Z}{\mathbb{Z}}
\newcommand{\Q}{\mathbb{Q}}
\newcommand{\R}{\mathbb{R}}
\newcommand{\C}{\mathbb{C}}

\newcommand{\A}{\mathbb{A}}

\renewcommand{\H}{\mathbb{H}}
\renewcommand{\P}{\mathbb{P}}
\newcommand{\Qal}{\overline{\mathbb{Q}}}

\renewcommand{\c}{\mathcal{C}}
\newcommand{\cc}{\mathbf{c}}

\renewcommand{\a}{\alpha}
\renewcommand{\b}{\beta}
\newcommand{\g}{\gamma}
\renewcommand{\d}{\delta}
\newcommand{\f}{\varphi}
\newcommand{\eps}{\varepsilon}
\renewcommand{\l}{\lambda}
\newcommand{\s}{\sigma}
\newcommand{\m}{\mu}

\renewcommand{\r}{\rho}

\renewcommand{\i}{\mathbf{i}}

\newcommand{\pt}[1]{\left( #1 \right)}
\newcommand{\pq}[1]{\left[ #1 \right]}
\newcommand{\pg}[1]{\left\lbrace #1 \right\rbrace}

\newcommand{\abs}[1]{\left| #1 \right|}

\renewcommand{\Im}{\mathrm{Im}}
\renewcommand{\Re}{\mathrm{Re}}

\newcommand{\quotes}[1]{``#1''}

\DeclareMathOperator{\en}{End}

\DeclareMathOperator{\Gal}{Gal}
\DeclareMathOperator{\GL}{GL}

\newtheorem{theorem}{Theorem}[section]
\newtheorem{proposition}[theorem]{Proposition}

\newtheorem{lemma}[theorem]{Lemma}
\theoremstyle{remark}
\newtheorem{remark}[theorem]{Remark}
\theoremstyle{definition}

\theoremstyle{definition}
\newtheorem{definition}[theorem]{Definition}

\usepackage[dvipsnames]{xcolor} 

\author[L. Ferrigno]{Luca Ferrigno}
\address{Dipartimento di Matematica e Fisica, Università degli Studi Roma Tre, Largo S. L. Murialdo 1, 00146 Roma, Italy}
\email{lucaferrigno.math@gmail.com}
\title{Isogeny relations in products of families of elliptic curves}

\begin{document}

\begin{abstract}
Let $E_{\l}$ be the Legendre family of elliptic curves with equation $Y^2=X(X-1)(X-\l)$. Given a curve $\c$, satisfying a condition on the degrees of some of its coordinates and parametrizing $m$ points $P_1, \ldots, P_m \in E_{\l}$ and $n$ points $Q_1, \ldots, Q_n \in E_{\m}$ and assuming that those points are generically linearly independent over the generic endomorphism ring, we prove that there are at most finitely many points $\cc_0$ on $\c$, such that there exists an isogeny $\phi: E_{\m(\cc_0)} \rightarrow E_{\l(\cc_0)}$ and the $m+n$ points $P_1(\cc_0), \ldots, P_m(\cc_0), \phi(Q_1(\cc_0)), \ldots, \phi(Q_n(\cc_0)) \in E_{\l(\cc_0)}$ are linearly dependent over $\en(E_{\l(\cc_0)})$.
\end{abstract}
\maketitle

\tableofcontents

\section{Introduction}
Let $m$ and $n$ be positive integers. Denote by $E_{\l}$ the elliptic curve with Legendre equation
$$Y^2Z=X(X-Z)(X-\l Z)$$ 
and consider this as a family of elliptic curves $E_{\l} \rightarrow Y(2)=\A^1 \setminus \pg{0,1}$. With a slight abuse of notation, we will denote by $E_{\l}^m$ the $m$-fold fibered power $E_{\l}\times_{Y(2)} \ldots \times_{Y(2)} E_{\l}$, which defines another family $E_{\l}^m \rightarrow Y(2)$. In this article we will work with the product 
$$E_{\l}^m \times E_{\m}^n \xrightarrow{\enspace\pi\enspace} Y(2) \times Y(2).$$
Here, $E_{\m} \rightarrow Y(2)$ is the Legendre family with parameter $\m$. 

Take an irreducible curve $\c \subseteq E_{\l}^m \times E_{\m}^n$, defined over a number field $k$, not contained in a fixed fiber. Then, for each point $\cc \in \c$, let $\pi(\cc)=(\l(\cc), \m(\cc)) \in Y(2) \times Y(2)$, where $\l$ and $\m$ are the coordinate functions on $Y(2)^2$. Also, $\cc \in \c$ defines $m$ points $P_1(\cc), \ldots, P_m(\cc)$ on the elliptic curve $E_{\l(\cc)}$ and $n$ points $Q_1(\cc), \ldots, Q_n(\cc)$ on the elliptic curve $E_{\m(\cc)}$. Let $R_1$ and $R_2$ denote the generic endomorphism rings of $E_{\l}$ and $E_{\m}$ when restricted to $\c$, respectively. In general, these are equal to $\Z$, except in the case where one of the elliptic curves is constant on $\c$ and has complex multiplication.  
For example, if $\l = \l_0$ is constant on $\c$ and $E_{\l_0}$ has complex multiplication, then $R_1$ is strictly larger than $\Z$.

We will assume that $E_{\l}$ and $E_{\m}$ are not generically isogenous on $\c$ and that the $P_i$ are linearly independent over $R_1$ and similarly for the $Q_i$. This is of course equivalent to saying that there are no generic non-trivial linear relations between the $P_i$ and the $Q_i$. Another way of rephrasing this is to say that $\c$ is not contained in a proper subgroup scheme of $E_{\l}^m \times E_{\m}^n \rightarrow Y(2) \times Y(2)$, again assuming that $E_{\l}$ and $E_{\m}$ are not generically isogenous on $\c$.

We define the map 
\begin{align*}
J:Y(2) &\longrightarrow Y(1)=\A^1\\
\l &\longmapsto 2^8 \dfrac{(\l^2-\l+1)^3}{\l^2 (\l-1)^2}
\end{align*}
which sends $\l$ to the $j$-invariant of $E_{\l}$. With a slight abuse of notation, we will also denote by $J$ the map $Y(2)^2 \rightarrow Y(1)^2$ obtained by applying $J$ component-wise.

\begin{definition}
Let $C \subseteq \A^{2}$ be an irreducible curve and let $X,Y$ be the coordinate functions on $\A^2$. We say that $C$ is \emph{asymmetric} (see \cite{Hab10}) if $\deg(X\vert_{C})\neq \deg(Y\vert_{C})$. Here, by convention, we set the degree of a constant map to be 0.

If $\c \subseteq E_{\l}^m \times E_{\m}^n \xrightarrow{\enspace\pi\enspace} Y(2) \times Y(2)$ is an irreducible curve, we say that $\c$ is \emph{asymmetric} if the curve $\widetilde{\c}=(J \circ \pi)(\c) \subseteq \A^2 $ is asymmetric.
\end{definition}

We are now ready to state the main result of this article.

\begin{theorem}\label{thm:main_thm_isog}
Let $\c \subseteq E_{\l}^m \times E_{\m}^n$ be an irreducible asymmetric curve defined over $\Qal$ not contained in a fixed fiber, and define $P_i, Q_j$ as above.
Suppose moreover that $E_{\l}$ and $E_{\m}$ are not generically isogenous on $\c$ and that there are no generic non-trivial relations among $P_1, \ldots, P_m$ on $E_{\l}$ and among $Q_1, \ldots, Q_n$ on $E_{\m}$ with coefficients in $R_1$ and $R_2$, respectively. Then, there are at most finitely many $\cc \in \c (\C)$ such that there exist an isogeny $\phi: E_{\m(\cc)} \rightarrow E_{\l(\cc)}$ and $(a_1, \ldots, a_{m+n}) \in \en(E_{\l(\cc)})^{m+n}\setminus \pg{\mathbf{0}}$ with $$a_1 P_1(\cc)+ \ldots +a_m P_m(\cc)+ a_{m+1} \phi \pt{Q_1(\cc)}+ \ldots + a_{m+n} \phi \pt{Q_n(\cc)}=O.$$
\end{theorem}

Notice that this theorem is a special case of the Zilber--Pink Conjecture. 
In combination with the results of \cite{BC16}, \cite{BC17}, \cite{Bar19}, and \cite{HP16}, and including the case in which one factor has complex multiplication and a linear relation holds among the points on the other factor (which will be addressed in future work), it yields a proof of the conjecture for asymmetric curves in $E_{\l}^m \times E_{\m}^n$ defined over $\Qal$.
For an account on the Zilber--Pink conjecture and other problems of Unlikely Intersections, see \cite{Zan12} and \cite{Pil22}.
\begin{remark}
Notice that if $E_{\l}$ and $E_{\m}$ are generically isogenous, then $\widetilde{\c}=Y_0(N)$ (for some $N\geq 1$) which is not asymmetric, since the modular polynomials are symmetric (see Section \ref{sect_isog}), and therefore have equal degree in both variables. Thus, in principle, the assumption that $E_{\l}$ and $E_{\m}$ are not generically isogenous could be removed from the theorem. 
However, in view of the Zilber–Pink conjecture, we expect that the theorem should remain valid even without the asymmetry assumption. For these reasons, in anticipation of a possible generalization of this result beyond the asymmetric setting, we prefer to leave the statement as it is.
\end{remark}

Depending on $\pi(\c) \subseteq Y(2)^2$, we can distinguish three cases: 
\begin{enumerate}[label=(\roman*)]
\item\label{case_1} the coordinate functions $\l, \m$ on $\c$ are both non-constant;
\item\label{case_2} (exactly) one between $\l$ and $\m$ is constant and the associated elliptic curve is not CM;
\item\label{case_3} (exactly) one between $\l$ and $\m$ is constant and the associated elliptic curve is CM.
\end{enumerate}

For each $\cc \in \c(\C)$, let $\r(\cc) \in \C$ be such that $\en(E_{\l(\cc)})\cong\Z\pq{\r(\cc)}$. 

In case \ref{case_1}, by a theorem by André \cite{And98}, there are only finitely many $\cc \in \c(\overline{\Q})$ such that $E_{\l(\cc)}$ and $E_{\m(\cc)}$ both have complex multiplication. So, recalling that isogenous elliptic curves have the same endomorphism algebra, we can discard those finitely many points and assume that $\r=0$ and $\mathbf{a} \in \Z^{m+n}\setminus \pg{\mathbf{0}}$.

Similarly, in case \ref{case_2}, we can assume without loss of generality that $\l=\l_0$ is constant with $E_{\l_0}$ not CM. Therefore, there are no points $\cc \in \c(\overline{\Q})$ such that $E_{\l(\cc)}$ and $E_{\m(\cc)}$ both have complex multiplication, so we can take $\r=0$ and $\mathbf{a} \in \Z^{m+n}\setminus \pg{\mathbf{0}}$ in this case as well.

In case \ref{case_3}, we can assume again that $\l=\l_0$ is constant. However, in this case there are infinitely many points $\cc \in \c(\overline{\Q})$ such that $E_{\l(\cc)}=E_{\l_0}$ and $E_{\m(\cc)}$ are both CM, so we cannot simplify our hypothesis as before. On the other hand, since $\l$ is constant, we can choose $\r$ to be a generator of $\en(E_{\l_0})\cong\Z[\r]$.

Our proof of Theorem \ref{thm:main_thm_isog} follows the general strategy first introduced by Pila and Zannier in \cite{PZ08} and later used, among others, by Masser and Zannier \cite{MZ10, MZ12} and by Barroero and Capuano \cite{BC16, Bar19,  BC17, BC20}. In what follows, we sketch the argument only in case \ref{case_1}; the proofs of cases \ref{case_2} and \ref{case_3} rely on the same strategy, although their implementation requires additional technical refinements.

Since the elliptic curves $E_{\l}$ and $E_{\m}$ are analytically isomorphic to the complex tori $\C/\Lambda_{\tau_1}$ and $\C/\Lambda_{\tau_2}$, where $\Lambda_{\tau}=\Z+\Z\tau$, with $\tau$ in the complex upper half-plane $\H$, we can  consider the elliptic logarithms $z_1 \ldots, z_m$ of $P_1, \ldots, P_m$ and $w_1, \ldots, w_n$ of $Q_1, \ldots, Q_n$ and define a uniformization map 
$$(\tau_1,z_1 \ldots, z_m, \tau_2,w_1, \ldots, w_n) \mapsto (\l, P_1, \ldots, P_m, \m, Q_1, \ldots, Q_n).$$
By work of Peterzil and Starchenko, after restricting to a suitable fundamental domain, this map is definable in the o-minimal structure $\R_{\text{an, exp}}$, so the preimage of $\c$ is a definable surface $S$.

Let $\c'$ be the subset of $\c$ we want to prove to be finite. Then, the points $\cc_0 \in \c'$ correspond to points on $S$ lying on subvarieties defined by equations with integer coefficients. We then use a result by Habegger and Pila, which implies that there are $\ll T^{\eps}$ points of $S$ lying on the subvarieties with coefficients bounded in absolute value by $T$, provided that the $z_i$ and the $w_j$ are algebraically independent over $\C(\tau_1, \tau_2)$. 

We then use a result by Habegger \cite{Hab10} for asymmetric curves\footnote{This is the only step of the proof where we use the assumption on the asymmetry of $\c$, see also Remark \ref{rmk_asym}.}, giving height bounds for $\l(\cc_0)$, $\m(\cc_0)$, the $P_i(\cc_0)$ and the $Q_j(\cc_0)$. 
By a result of Masser \cite{Mas88}, these bounds imply that the coefficients $a_1, \ldots, a_{m+n}$ of the linear relation between  the $m+n$ points 
$$P_1(\cc_0), \ldots, P_m(\cc_0), \phi(Q_1(\cc_0)), \ldots, \phi(Q_n(\cc_0))$$
can be taken to be bounded by a constant times a positive power of $D_0=\pq{k(\l(\cc_0),\m(\cc_0)):k}$. Moreover, all Galois conjugates of $\cc_0$ are still in $\c'$, so that we have at least $D_0$ points on $S$ lying on the subvarieties with coefficients bounded in absolute value by some positive power of $D_0$. Combining this with the previous bound, we get that $D_0$ is bounded and therefore the claim of the theorem, by Northcott's theorem.

\begin{remark}\label{rmk:pti_const}
Let $C \subseteq E_{\l}^m \times E_{\m}^n$ be an irreducible curve not contained in a fixed fiber, not necessarily asymmetric, and define $P_i, Q_j$ and $\tau_1, z_1, \ldots, z_m, \tau_2, w_1, \ldots, w_n$ as above (see also Section \ref{sec_unif} for more details).
Assume also that $E_{\l}$ and $E_{\m}$ are not generically isogenous on $C$ and that there are no generic non-trivial relations among $P_1, \ldots, P_m$ on $E_{\l}$ and among $Q_1, \ldots, Q_n$ on $E_{\m}$ with coefficients in $R_1$ and $R_2$, respectively.

Then, in case \ref{case_1} let $\ell=0$, while in cases \ref{case_2} and \ref{case_3} let $\ell \geq 0$ be the greatest integer such that there are $\widetilde{a}_{i,j} \in \en(E_{\l_0})$ and $\widetilde{P}_j \in E_{\l_0}(\Qal)$, $i=1, \ldots, m$ and $j=1, \ldots, \ell$, such that the vectors $\widetilde{\mathbf{a}}_{j}:=(\widetilde{a}_{1,j}, \ldots, \widetilde{a}_{m,j})$, for $j=1, \ldots, \ell$, are $\en(E_{\l_0})$-linearly independent and
$$\widetilde{a}_{1,j} P_1(\cc) + \ldots + \widetilde{a}_{m,j} P_m(\cc)=\widetilde{P}_j$$
for every $\cc \in C$.
In particular, the assumption that there is no generic non-trivial relation among the $P_i$ and the assumption on the vectors $\widetilde{\mathbf{a}}_{1}, \ldots, \widetilde{\mathbf{a}}_{\ell}$ implies that $\widetilde{P}_1, \ldots, \widetilde{P}_{\ell}$ are $\en(E_{\l_0})$-linearly independent.
Up to reordering the $P_i$, we can then assume that the matrix $(\widetilde{a}_{i,j})_{i,j=1, \ldots, \ell}$ has maximal rank. 
Hence, we can consider the isogeny
\begin{align*}
\Phi: E_{\l_0}^m \times E_{\m}^n &\longrightarrow E_{\l_0}^m \times E_{\m}^n\\
(P_1, \ldots, P_m, Q_1, \ldots, Q_n) &\longmapsto \pt{\sum\limits_{i=1}^m \widetilde{a}_{i,1}P_i, \ldots, \sum\limits_{i=1}^m \widetilde{a}_{i,\ell}P_i, P_{\ell+1}, \ldots, P_m, Q_1, \ldots, Q_n}
\end{align*}
which sends $(\mathbf{P}, \mathbf{Q}) \in C$ to $(\widetilde{P}_1, \ldots, \widetilde{P}_{\ell}, P_{\ell+1}, \ldots, P_m, \mathbf{Q}) \in \Phi(C)$. Note that the maximality of $\ell$ also implies that $P_{\ell+1}, \ldots, P_m$ are generically $\en(E_{\l_0})$-linearly independent modulo constants, i.e.\ there is no relation of the form 
$$a_{\ell+1} P_{\ell+1}(\cc) + \ldots + a_m P_m(\cc) = \widetilde{P}$$
with $a_{\ell+1}, \ldots, a_m \in \en(E_{\l_0})$ not all zero and $\widetilde{P} \in E_{\l_0}(\Qal)$, that holds for every $\cc \in C$.

Notice that if $C$ satisfies the hypotheses of Theorem \ref{thm:main_thm_isog}, then $\Phi(C)$ satisfies them as well, and vice versa. Moreover, as the restriction of $\Phi$ to any fiber is again an isogeny, images and preimages under $\Phi$ of algebraic subgroups of a fiber are again algebraic subgroups. This implies that Theorem \ref{thm:main_thm_isog} holds for $C$ if and only if it holds for $\Phi(C)$.

Therefore, up to applying the isogeny $\Phi$, we will always assume that on the (asymmetric) curve $\c$ that we are considering in Theorem \ref{thm:main_thm_isog}, $P_1, \ldots, P_{\ell}$ are constant and $\en(E_{\l_0})$-linearly independent, as above.
\end{remark}

\begin{remark}
I am grateful to Gabriel Dill for pointing out that cases \ref{case_2} and \ref{case_3} can also be deduced from Theorem 1.2 in \cite{Dil21}. 
Continuing with the notation introduced in the previous remark and using the notation from \cite{Dil21}, take $A_0=E_{\l_0}^{m-\ell+n}$, $\mathcal{A}=E_{\l_0}^{m-\ell} \times E_{\m}^n$ and $\Gamma=(\Gamma_0)^{m-\ell+n}$, where $\Gamma_0$ is the divisible hull of the subgroup of $E_{\l_0}(\Qal)$ generated by $\en(E_{\l_0})\cdot \widetilde{P}_1, \ldots, \en(E_{\l_0})\cdot \widetilde{P}_{\ell}$. 
Then, if $C$ is the projection of $\Phi(\c)$ onto $\mathcal{A}$, $\mathcal{A}_{\Gamma}^{\pq{1}} \cap C$ consists exactly of the points described in Theorem \ref{thm:main_thm_isog} and, by \cite[Theorem 1.2]{Dil21}, we get that either this intersection is finite or that the generic fiber $C_{\xi} \subset C$ is contained in the translate of a proper abelian subvariety of $\mathcal{A}_{\xi}$ by a point in 
$$(\mathcal{A}_{\xi})_{\mathrm{tors}}+ \mathrm{Tr}(\mathcal{A}_{\xi})=E_{\l_0}^{m-\ell}(\Qal) \times (E_{\m}^n)_{\mathrm{tors}}.$$ 
However, the latter means that either $Q_1, \ldots, Q_n$ are generically linearly dependent or that there is a non-trivial linear relation modulo constants involving $P_{\ell+1}, \ldots, P_m$, contradicting our assumptions. 
\end{remark}

In this article we use Vinogradov's $\ll$ notation: for real-valued functions $f(T)$ and $g(T)$, we write $f(T) \ll g(T)$ if there exists a constant $\gamma > 0$ such that $f(T) \leq \gamma g(T)$ for all sufficiently large $T$. When not explicitly stated, the implied constant is either absolute or depends only on $\mathcal{C}$ and other fixed data. We use subscripts to indicate any additional dependence of the implied constant. 

\section{Preliminaries}

\subsection{Isogenies and modular curves}\label{sect_isog}
Let $E_1 \cong \C/\Lambda_1$ and $E_2 \cong \C/\Lambda_2$ be two elliptic curves defined over $\C$. Up to homothety, any lattice in $\C$ is of the form $\Z + \Z\tau$ for some $\tau$ in the upper half-plane $\mathbb{H}$, and any two such lattices define isomorphic complex tori if and only if the corresponding parameters $\tau$ lie in the same orbit under the action of $\mathrm{SL}_2(\Z)$ on $\mathbb{H}$ by fractional linear transformations. Therefore, we may choose $\tau_1, \tau_2 \in \mathbb{H}$ such that $\Lambda_1 = \Z + \Z\tau_1$ and $\Lambda_2 = \Z + \Z\tau_2$, with $\tau_1$ and $\tau_2$ lying in the standard fundamental domain $\mathfrak{F} \subseteq \mathbb{H}$ for the action of $\mathrm{SL}_2(\Z)$. This domain is given by
\begin{equation}\label{eqn:def_fund_dom}
\mathfrak{F}= \pg{\tau \in \mathbb{H}: \abs{\tau}\geq 1, -\frac{1}{2} \leq \Re(\tau) < \frac{1}{2}} \setminus \pg{\tau \in \mathbb{H}: \abs{\tau}= 1, 0 < \Re(\tau) < \frac{1}{2}}.
\end{equation}
In particular, the choice of $\tau_1, \tau_2 \in \mathfrak{F}$ ensures that the associated elliptic curves are uniquely determined up to isomorphism.

Recall that for each isogeny $ \phi:E_1 \rightarrow E_2 $ there exists a unique non-zero complex number $ \a $ such that $ \a \Lambda_1 \subseteq \Lambda_2 $ and $\phi$ corresponds to the multiplication-by-$\a$ map $ \C/\Lambda_1 \rightarrow \C/\Lambda_2$.

Therefore, if $ E_1 $ and $ E_2 $ are isogenous, then there exist $ \a \in \C\setminus\pg{0}$ and integers $ A,B,C,D $  not all zero (not necessarily coprime) such that
\begin{align*}
	\a\cdot \tau_1&= A\tau_2+B\\
	\a\cdot 1&= C\tau_2+D
\end{align*}
thus \[ \tau_1 =\dfrac{A\tau_2+B}{C\tau_2+D}. \]
Moreover, the converse is also true. If $\tau_1, \tau_2 \in \H$ and $\tau_1 =\frac{A\tau_2+B}{C\tau_2+D}$ for integers $A,B,C,D$, then there exists an isogeny $ \phi:E_1\rightarrow E_2$, corresponding to $ \a=C\tau_2+D$.

More generally, we have an action of the group $\GL_2^+(\Q)$ (here $+$ means that the matrices have positive determinant) on the upper half-plane $\H$ which is given by
$$ M\tau=\dfrac{a\tau+b}{c\tau+d}$$ for $M=\begin{psmallmatrix}
a & b \\ 
c & d
\end{psmallmatrix}\in \GL_2^+(\Q)$. If $M \in \textrm{Mat}(\Z,2)$, we say that $M$ is \emph{primitive} if $\gcd(a,b,c,d)=1$. 

We say that an isogeny $\phi$ is \emph{cyclic} if $\ker \phi$ is a (finite) cyclic group. Then it is known (see \cite[Section 1.3]{DS05}) that any isogeny can be written as the composition of a cyclic isogeny and a multiplication-by-$n$ isogeny, for some integer $n$. In particular, cyclic isogenies $E_1 \rightarrow E_2$ correspond to relations $\tau_1 =M\tau_2$ with $M$ primitive. In this case, the degree of the isogeny is equal to $\det M$.

Recall also that the modular polynomials  $\Phi_N(X,Y) \in \Z\pq{X,Y}$ are the irreducible symmetric polynomials parametrizing pairs of isomorphism classes of elliptic curves with a cyclic isogeny of degree $N$ between them \cite[Chapter 5]{Lan87}. In other words, $\Phi_N(j_1,j_2)=0$ if and only if there exists a cyclic isogeny of degree $N$ between the elliptic curves with $j$-invariants $j_1$ and $j_2$. 
We then define the classical modular curve $Y_0(N) \subset \A^2$ as the plane curve defined by the equation $\Phi_N(X,Y)=0$.

Finally, the following result provides an effective bound for the size of the integers $A, B, C, D$ when the degree of the isogeny is fixed. It is a consequence of Theorem 1.1 of \cite{Orr18} (see Section 1.A therein for details), which constitutes an improvement of both Lemma 2.1 of \cite{MZ20b} and Lemma 5.2 of \cite{HP12}.

\begin{lemma}\label{lemma_bound_isogenie_matrix}
There exists an absolute constant $c > 0$ with the following property: if 
$$E_1=\C/(\Z+\Z\tau_1), \qquad E_2=\C/(\Z+\Z\tau_2)$$
are elliptic curves with $\tau_1, \tau_2 \in \mathfrak{F}$, and there exists a cyclic isogeny $ \phi: E_1 \rightarrow E_2 $ of degree $N$, then there are integers $ A,B,C,D $ such that 
\[ AD-BC=N \qquad \tau_1 =\dfrac{A\tau_2+B}{C\tau_2+D} \qquad \abs{A}, \abs{B},\abs{C}, \abs{D}\leq c N. \]
\end{lemma}

\subsection{Uniformization}\label{sec_unif}
Let $\mathcal{A}$ be the quasi-projective variety in $Y(2)\times (\P^2)^m\times Y(2) \times (\P^2)^n$ with coordinates 
$$(\l, \pq{X_1:Y_1:Z_1}, \ldots, \pq{X_m:Y_m:Z_m}, \m, \pq{U_1:V_1:W_1}, \ldots, \pq{U_n:V_n:W_n})$$
and defined by the $n+m$ equations
\begin{align*}
Y_i^2Z_i &=X_i(X_i-Z_i)(X_i-\l Z_i) & i=1, \ldots, m\\
V_j^2W_j &=U_j(U_j-W_j)(U_j-\m W_j) & j=1, \ldots, n.
\end{align*}
We set $P_i=\pq{X_i:Y_i:Z_i}$ and $Q_j=\pq{U_j:V_j:W_j}$ and we have an irreducible curve $C \subseteq \mathcal{A}$ defined over a number field $k$ such that the projection of $\mathcal{A}$ to $Y(2)\times Y(2)$ restricts to rational functions $\l$ and $\m$ on $\c$ not both constant.

The aim of this section is to define a uniformization map for $\mathcal{A}$, following closely the exposition in \cite[pp.~2489–2491]{Pil09b}.

As said before, any elliptic curve over $\C$ is analytically isomorphic to a complex torus $\C/\Lambda_{\tau}$, where $\tau$ has positive imaginary part and $\Lambda_{\tau}$ is the lattice generated by $1$ and $\tau$, with fundamental domain
$$\mathcal{L}_{\tau}=\pg{z\in \C: z=x+\tau y, x,y \in \left[0, 1\right)}.$$
The classical Weierstrass $\wp$-function $\wp(z,\Lambda_{\tau})=\wp(z,\tau)$ associated to the lattice $\Lambda_{\tau}$, is $\Lambda_{\tau}$-periodic and satisfies the following differential equation
$$ (\wp(z,\tau)')^2=4\wp(z,\tau)^3-g_2(\tau)\wp(z,\tau)-g_3(\tau) $$ 
where $\wp(z,\tau)'=\frac{d}{dz}\wp(z,\tau)$ and $g_2(\tau), g_3(\tau)$ are defined in \cite[Remark 3.5.1]{Sil09}. Then, the zeros of the polynomial $ 4X^3-g_2(\tau)X-g_3(\tau) $ are exactly the values of $ \wp $ at the half-periods:
$$ e_1(\tau)=\wp\left(\dfrac{1}{2},\tau \right) \quad e_2(\tau)=\wp\left(\dfrac{1+\tau}{2},\tau \right) \quad e_3(\tau)=\wp\left(\dfrac{\tau}{2},\tau \right).$$
Note that the $ e_i(\tau) $ are pairwise distinct (see \cite[Proposition VI.3.6]{Sil09} and \cite[Section 63]{For51}) and that the function $e_3-e_1$ has a regular square root on all of $\mathbb{H}$. Therefore, we can define 
$$ \xi(z,\tau)=\dfrac{\wp(z,\tau)-e_1(\tau)}{e_3(\tau)-e_1(\tau)} \quad \text{ and } \quad  \eta(z,\tau)=\dfrac{\wp(z,\tau)'}{2(e_3(\tau)-e_1(\tau))^{\frac{3}{2}}}$$
so that we have the following relation
$$ \eta(z,\tau)^2=\xi(z,\tau)(\xi(z,\tau)-1)(\xi(z,\tau)-L(\tau)) $$
where 
\begin{equation} \label{def_L}
L(\tau)=\dfrac{e_2(\tau)-e_1(\tau)}{e_3(\tau)-e_1(\tau)}.
\end{equation}
This gives a parametrization of the Legendre family via the map $ (z,\tau) \mapsto \left(L(\tau), P(z,\tau) \right)$, where
$$ P(z,\tau)=\begin{cases}
\left[\xi(z,\tau): \eta(z,\tau):1 \right] &\text{ if } z \not\in \Lambda_{\tau}\\
\left[0:1:0\right] &\text{ otherwise }
\end{cases} $$
Finally, define the map 
$$\f: \mathbb{H}\times \C^m\times \mathbb{H} \times \C^n \rightarrow \mathcal{A}(\C)$$
sending $\pt{\tau_1,z_1,\ldots, z_m, \tau_2, w_1, \ldots, w_n}$ to 
$$\pt{L(\tau_1), P(z_1,\tau_1), \ldots, P(z_m, \tau_1),L(\tau_2), P(w_1, \tau_2), \ldots, P(w_n,\tau_2)}.$$
Since this map is not injective, we would like to find a subset of the domain over which it is possible to define a univalued inverse function of $\f$.

By \cite[Section 70]{For51}, there exists a finite index subgroup $\Gamma$ of $\mathrm{SL}_2(\Z)$ such that $L(\gamma \tau)=L(\tau)$ for every $\gamma \in \Gamma$. Moreover, as a fundamental domain for the action of $\Gamma$ on $\mathbb{H}$ one can take the union of six suitably chosen fundamental domains for the action of $\mathrm{SL}_2(\Z)$ (see \cite[Figures 48 and 49]{For51}). We will call this set $\mathcal{B}$ and define
$$\mathcal{F}_{\mathcal{B}}=\pg{\pt{\tau_1,z_1,\ldots, z_m, \tau_2, w_1, \ldots, w_n}: \tau_1, \tau_2 \in \mathcal{B}, z_1, \ldots, z_m \in \mathcal{L}_{\tau_1}, w_1, \ldots, w_n \in \mathcal{L}_{\tau_2}}.$$

Then, $\f$ has a univalued inverse $\mathcal{A}(\C) \rightarrow \mathcal{F}_{\mathcal{B}}$ and we set
\begin{equation}\label{def:Z_isog}
\mathcal{Z}=\f^{-1}(\c(\C)) \cap \mathcal{F}_{\mathcal{B}}.
\end{equation}
Following Remark \ref{rmk:pti_const}, we assume that $P_1=\widetilde{P}_1, \ldots, P_{\ell}=\widetilde{P}_{\ell}$ are constant on $\c$, so that $\mathcal{Z}$ consists of points of the form $\pt{\tau_1,\widetilde{z}_1,\ldots,\widetilde{z}_{\ell}, z_{\ell+1}, \ldots, z_m, \tau_2, w_1, \ldots, w_n}$, where $\widetilde{z}_i$ is the (constant) elliptic logarithm of the constant point $\widetilde{P}_i$.

Having described the uniformization of $\mathcal{A}$, we now turn to a key result of functional transcendence that will play an essential role in our arguments. 
Let $\widehat{\c}$ be the subset of the smooth points of $\c(\C)$ that are not ramified points of $\pi_{\vert_{\c}}$. In this way, the set $\c(\C) \setminus \widehat{\c}$ consists only of finitely many algebraic points of $\c$.

Fix a point $\cc_{*} \in \widehat{\c}$ and let $D_{\cc_{*}}\subseteq \widehat{\c}$ be a small open disc containing $\cc_{*}$. Let $\mathbf{t}_{*}=\f_{\vert\mathcal{F}_{\mathcal{B}}}^{-1}(\cc_{*})$. Then, there exists an open connected neighbourhood $U_{*} \subseteq \mathbb{H}\times \C^m\times \mathbb{H} \times \C^n$ of $\mathbf{t}_{*}$ such that $\f(U_{*})=D_{\cc_{*}}$.
Thus, it follows that $\tau_1, z_1, \ldots, z_m, \tau_2, w_1, \ldots, w_n$ are well-defined holomorphic functions (possibly constant) on $U_{*}$ and, with a slight abuse of notation, we consider them as holomorphic functions on $D_{\cc_{*}}$.

With these definitions, we have the following transcendence result.
\begin{lemma}\label{lemma:func_transc_isog}
The functions $z_{\ell+1},\ldots,z_m, w_1, \ldots, w_n$ are algebraically independent over $\C(\tau_1, \tau_2)$ on $D_{\cc_{*}}$.
\end{lemma}
\begin{proof}
In case \ref{case_1} we can apply Corollary 2.5 from \cite{BC17} (which is based on a result by Bertrand \cite{Ber09}).

In cases \ref{case_2} and \ref{case_3}, notice that our assumptions on $\c$ imply that 
$$\c \subseteq \pg{\pt{\widetilde{P}_1, \ldots, \widetilde{P}_{\ell}}} \times E_{\l_0}^{m-\ell} \times E_{\m}^n.$$ 
Let also $F=\C\pt{\tau_2}$ (note that $F=\C(\tau_1, \tau_2)$, since $\tau_1$ is constant in these cases) and assume by contradiction that 
$$\mathrm{tr.deg}_F F\pt{z_{\ell+1}, \ldots, z_m, w_1, \ldots, w_n} < m+n-\ell.$$
Then, if $\widetilde{z}_1, \ldots, \widetilde{z}_{\ell} \in \mathcal{L}_{\tau_1}$ are elliptic logarithms of the constant points $\widetilde{P}_1, \ldots, \widetilde{P}_{\ell}$, on $\c$ we must have
$$\mathrm{tr.deg}_F F(\widetilde{z}_1, \ldots, \widetilde{z}_{\ell}, z_{\ell+1}, \ldots, z_m, w_1, \ldots, w_n ) =\mathrm{tr.deg}_F F\pt{z_{\ell+1}, \ldots, z_m, w_1, \ldots, w_n} < m+n-\ell.$$
Applying Theorem 7.1 from \cite{Dil21} to $\c$, we obtain a subvariety $\mathcal{W}\subseteq E_{\l_0}^m\times E_{\m}^n$ containing $\c$. This subvariety is a translate of an abelian subscheme of $E_{\l_0}^m \times E_{\m}^n$ by the product of a constant section of $E_{\l_0}^m$ (defined over $\Qal$) with a torsion multisection of $E_{\m}^n$. Moreover, one has $\dim(\mathcal{W}) \leq m+n-\ell$.

Since $Q_1, \ldots, Q_n$ are linearly independent by hypothesis, this implies that there are $\widetilde{a}_{i,\ell+1} \in \en(E_{\l_0})$ and $\widetilde{P}_{\ell+1} \in E_{\l_0}(\Qal)$, $i=\ell+1, \ldots, m$, such that $\widetilde{a}_{\ell+1,\ell+1}, \ldots, \widetilde{a}_{m,\ell+1}$ are not all zero and
$$\widetilde{a}_{\ell+1,\ell+1} P_{\ell+1} + \ldots + \widetilde{a}_{m,\ell+1} P_m=\widetilde{P}_{\ell+1}$$ 
contradicting the maximality of $\ell$ and proving that
\setlength{\belowdisplayskip}{-1em}
$$\mathrm{tr.deg}_F F\pt{z_{\ell+1}, \ldots, z_m, w_1, \ldots, w_n} = m+n-\ell.$$ 
\end{proof}

\subsection{Heights}
Let $h$ denote the logarithmic absolute Weil height on $\P^N$, as defined in \cite{BG06} and, if $\a$ is an algebraic number, define $h(\a)=h\pt{\pq{1:\a}}$. Define also the multiplicative Weil height as $H(P)=\exp(h(P))$.

For an elliptic curve $E$ defined over $\Qal$ and a point $P \in E(\Qal)\subseteq \P^2(\Qal)$, we also have the Néron-Tate height $\widehat{h}$, defined as follows (see also \cite[VII.9]{Sil09}):
$$\widehat{h}(P)=\lim\limits_{n\rightarrow\infty}\dfrac{1}{4^n}h(2^nP).$$
\begin{proposition}\label{PropHIso}
Let $E_1,E_2$ two elliptic curves defined over $\Qal$ and let $\phi: E_1 \rightarrow E_2$ be an isogeny also defined over $\Qal$. Denote by $\widehat{h}_1$ and $\widehat{h}_2$ the Néron-Tate heights on $E_1$ and $E_2$, respectively. Then for any $P \in E_1(\Qal)$, we have
$$\widehat{h}_2(\phi(P))=\deg \phi \cdot \widehat{h}_1(P).$$
\end{proposition}
\begin{proof}
Recall that the height $h$, given by the embedding of an elliptic curve $E$ into $\P^2$, is exactly the height associated to the divisor $3(O)$ as described in \cite[Section B.3]{HS13}, where $O$ is the identity element on $E$. Moreover, if $O_i$ is the identity on $E_i$ and $\phi : E_1 \rightarrow E_2$ is an isogeny, we have 
$$\phi^* 2(O_2)=2\sum\limits_{T \in \ker \phi}(T)\sim 2(T')+(2\deg\phi-2)(O_1)$$
where $T'=\sum\limits_{T \in \ker \phi}T$ is either $O_1$ or a non-trivial 2-torsion point. In either case, $2(T')\sim 2(O_1)$, so that the pull-back $\phi^* 2(O_2)$ is linearly equivalent to $(2\deg \phi)(O_1)$.
Then, using \cite[Theorem B.5.6]{HS13}, we get
\begin{align*}
2 \widehat{h}_2(\phi(P))&=2 \widehat{h}_{E_2, 3(O_2)}(\phi(P))= \widehat{h}_{E_2, 6(O_2)}(\phi(P))\\
	&= 3 \widehat{h}_{E_2, 2(O_2)}(\phi(P))= 3 \widehat{h}_{E_1, \phi^* 2(O_2)}(P)=\\
	&=3 \widehat{h}_{E_1, (2\deg\phi)(O_1)}(P) = 2 \deg \phi \cdot \widehat{h}_{E_1, 3(O_1)}(P) = 2 \deg \phi \cdot \widehat{h}_1(P)
\end{align*}
which is equivalent to $\widehat{h}_2(\phi(P))=\deg \phi \cdot \widehat{h}_1(P)$.
\end{proof}
Using the same notation as in the previous section, we have that if $\cc \in \c\pt{\Qal}$, then standard properties of heights (see \cite[Theorem B.2.5]{HS13}) imply that, if $\l$ and $\m$ are both non-constant, we have 
\begin{equation}\label{eqn:height_bounds_P_lambda}
h(P_i(\cc))\ll h(\l(\cc))+1 \quad \textrm{and} \quad  h(Q_j(\cc))\ll h(\m(\cc))+1
\end{equation}
for every $i=\ell+1,\ldots,m$ and $j=1, \ldots, n$. In cases \ref{case_2} and \ref{case_3}, if $\l=\l_0$ is constant, then $h(P_i(\cc))\ll h(\m(\cc))+1$, as we can use $\m$ as uniformizing parameter on the base $\pi(\c)=\pg{\l_0} \times \A^1$.
Moreover, note that if $\c$ is defined over a number field $k$, we also have 
$$\pq{k(\cc):k}\ll \pq{k(\l(\cc),\m(\cc)):k}.$$
Finally, we will also need another definition of height (from \cite{Pil09}).
\begin{definition}
If $\a$ is a complex number, we define
$$H_1(\a):=\begin{cases}
H(\a)=\max\pg{\abs{p},\abs{q}} \quad \text{ if } \a=\frac{p}{q} \in \Q, \gcd(p,q)=1\\
+\infty \quad \text{otherwise}
\end{cases}$$
For $(\a_1, \ldots, \a_N) \in \C^N$, we also define $H_1(\a_1, \ldots, \a_N)=\max\pg{H_1(\a_i)}$.
\end{definition}

\subsection{Complex Multiplication}
Given a $\l_0 \in Y(2)$ such that $E_{\l_0}$ has complex multiplication, we know that the associated $\tau_0 \in \mathcal{B}$ is an algebraic number of degree 2, with minimal polynomial $aX^2+bX+c \in \Z[X]$ and discriminant $\Delta_0=b^2-4ac <0$. In this case, we know by \cite[Theorem 1, p. 90]{Lan87}, that 
$$\en(E_{\l_0})\cong \Z\pq{\r_0}=:\mathcal{O}_{\l_0}$$
where $\r_0=\frac{\Delta_0+\sqrt{\Delta_0}}{2} \in \C$ and the isomorphism is given by \cite[Proposition II.1.1]{Sil94}. Using this proposition and with a slight abuse of notation, we will identify endomorphisms with the corresponding complex number.

By Theorem II.4.3. of \cite{Sil94},
$$\pq{\Q(j_0):\Q}=\mathrm{cl}(\mathcal{O}_{\l_0})$$
where $j_0$ is the $j$-invariant of $E_{\l_0}$ (which is algebraic by \cite[Proposition II.2.1]{Sil94}) and $\mathrm{cl}(\mathcal{O}_{\l_0})$ is the class number of $\mathcal{O}_{\l_0}$.

Moreover, a theorem of Siegel in the form of Theorem 1.2 of \cite{Bre01} gives us the estimate
$$\abs{\Delta_0}^{\frac{1}{2}-\epsilon}\ll_{\epsilon} \mathrm{cl}(\mathcal{O}_{\l_0}) \ll_{\epsilon} \abs{\Delta_0}^{\frac{1}{2}+\epsilon} $$
so that, in particular, we have $\abs{\Delta_0}\ll \pq{\Q(j_0):\Q}^3$. Finally, using Proposition \ref{PropHIso} and the fact that the endomorphism $\r_0$ has degree $\pt{\Delta_0^2-\Delta_0}/4$, we get that
\begin{equation}\label{diseq_alt_rhoP}
\widehat{h}(\r_0 P)\ll \abs{\Delta_0}^2 \widehat{h}(P) \ll \pq{\Q(j_0):\Q}^6 \widehat{h}(P) \ll \pq{\Q(\l_0):\Q}^6 \widehat{h}(P)
\end{equation}
for every $P \in E_{\l_0}\!\!\pt{\Qal}$.

\section{O-minimality and definable sets}\label{sect:o-min_isog}
In this section we recall the basic properties and some results about o-minimal structures. For more details see \cite{vdDri98} and \cite{vdDM96}.
\begin{definition}
A \emph{structure} is a sequence $\mathcal{S}=\pt{\mathcal{S}_N}$, $N\geq 1$, where each $\mathcal{S}_N$ is a collection of subsets of $\R^N$ such that, for each $N,M\geq 1$:
\begin{itemize}
\item $\mathcal{S}_N$ is a boolean algebra (under the usual set-theoretic operations);
\item $\mathcal{S}_N$ contains every semi-algebraic subset of $\R^N$;
\item if $A \in \mathcal{S}_N$ and $B \in \mathcal{S}_M$, then $A\times B \in \mathcal{S}_{N+M}$;
\item if $A\in \mathcal{S}_{M+N}$, then $\pi(A) \in \mathcal{S}_M$, where $\pi: \R^{M+N} \rightarrow \R^M$ is the projection onto the first $M$ coordinates.
\end{itemize}
If $\mathcal{S}$ is a structure and, in addition, 
\begin{itemize}
\item $\mathcal{S}_1$ consists of all finite unions of open intervals and points
\end{itemize}
then $\mathcal{S}$ is called an \emph{o-minimal structure}.
\end{definition}
Given a structure $\mathcal{S}$, we say that $S \subseteq \R^N$ is a \emph{definable set} if $S \in \mathcal{S}_N$. 

Given $S\subseteq \R^N$ and a function $f: S\rightarrow \R^M$, we say that $f$ is a \emph{definable function} if its graph $\pg{(x,y)\in \R^N\times \R^M: x\in S, y=f(x)}$ is a definable set. One can easily prove that images and preimages of definable sets via definable functions are still definable.

Let $U\subseteq \R^{M+N}$. For $t_0\in \R^M$, we set $U_{t_0}=\pg{x\in \R^N:(t_0,x)\in U}$ and call $U$ a \emph{family} of subsets of $\R^N$, while $U_{t_0}$ is called the \emph{fiber} of $U$ above $t_0$. If $U$ is a definable set, then we call it a \emph{definable family} and it is easy to prove that the fibers $U_{t_0}$ are also definable.

While there are many examples of o-minimal structures (see \cite{vdDM96}), in this article we will work with the structure $\R_{\text{an,exp}}$ (see \cite[Chapter 8]{Pil22} for details about this structure), which was proved to be o-minimal by van den Dries and Miller \cite{vdDM94}.

For a family $Z \subseteq \R^M \times \R^N= \R^{M+N}$ and a positive real number $T$ define
\begin{equation*}
Z^{\sim}(\Q,T):=\pg{(y,z)\in Z: y \in \Q^M, H_1(y)\leq T}
\end{equation*}
where $H_1(y)$ is the 1-polynomial height defined in the previous section and let $ \pi_1, \pi_2 $ be the projections of $ Z $ to the first $ M $ and last $ N $ coordinates, respectively.
\begin{proposition}[\cite{HP16}, Corollary 7.2]\label{prop:thm_HP_isog}
Let $ Z \subseteq \R^{M+N}$ be a definable set. For every $ \varepsilon>0 $ there exists a positive constant $ c=c(Z,\varepsilon) $ with the following property. If $ T\geq 1 $ and $ \lvert \pi_2(Z^{\sim}(\Q,T)) \rvert>cT^{\varepsilon} $, then there exists a continuous definable function $ \delta: \left[0,1 \right] \rightarrow Z  $ such that:
\begin{enumerate}
\item the restriction $\delta \vert_{(0,1)}$ is real analytic (since $\R_{\text{\emph{an, exp}}}$ admits analytic cell decomposition);
\item the composition $ \pi_1\circ \delta: \left[0,1 \right] \rightarrow \R^M$ is semi-algebraic and its restriction to $ (0,1) $ is real analytic;
\item the composition $ \pi_2\circ \delta: \left[0,1 \right] \rightarrow \R^N$ is non-constant.
\end{enumerate}
\end{proposition}

Lastly, we want to prove that the set $\mathcal{Z}$ defined in (\ref{def:Z_isog}) is definable in $\R_{\mathrm{an, exp}}$. In the following, definability will always be considered in $\R_{\mathrm{an, exp}}$, and we say that $X \subseteq \C^N$ is definable if the set 
$$\pg{\pt{\Re(z_1), \Im(z_1), \ldots, \Re(z_N),\Im(z_N)}: (z_1, \ldots, z_N) \in X} \subseteq \R^{2N}$$
is definable. Similarly, a function $f: X \rightarrow \C$ is definable if and only if $\Re(f)$ and $\Im(f)$ are both definable.

Let $\mathfrak{F}$ be the fundamental domain for the action of $\mathrm{SL}_2(\Z)$ on $\mathbb{H}$ defined in Section \ref{sect_isog}; then the restriction of $\wp(z, \tau)$ to $\pg{(z, \tau): \tau \in \mathfrak{F}, z \in \mathcal{L}_{\tau}}$ is definable by work of Peterzil and Starchenko \cite{PS05}. Therefore, $\wp(z, \tau)$ is definable even if restricted to $\pg{(z, \tau): \tau \in \gamma \mathfrak{F}, z \in \mathcal{L}_{\tau}}$, for any fundamental domain  $\gamma \mathfrak{F}$ for $\mathrm{SL}_2(\Z)$. Since $\mathcal{B}$ is the union of six such fundamental domains, we have that $\wp(z, \tau)$ is also definable when restricted to $\pg{(z, \tau): \tau \in \mathcal{B}, z \in \mathcal{L}_{\tau}}$. 
Thus, the uniformization map $\f$, defined in the previous section and restricted to $\mathcal{F}_{\mathcal{B}}$, is definable. Since $\c(\C)$ is semi-algebraic and $\mathcal{F}_{\mathcal{B}}$ is definable, we get that $\mathcal{Z}=\f^{-1}(\c(\C))\cap \mathcal{F}_{\mathcal{B}}$ is definable.

\section{The main estimate}
We continue with the notation established in the previous sections and, for every $T\geq 1$, define the set
\begin{align*}
\mathcal{Z}(T)=\Bigg\lbrace&\pt{\tau_1,\widetilde{z}_1,\ldots,\widetilde{z}_{\ell}, z_{\ell+1}, \ldots, z_m, \tau_2, w_1, \ldots, w_n} \in \mathcal{Z}: \lvert \tau_1\rvert, \lvert \tau_2 \rvert\leq T, \Im(\tau_1)\geq \dfrac{1}{T},\\
& \text{there exist } A,B,C,D \in \Z \cap \pq{-T,T} \text{ with } AD-BC\neq 0, \tau_2=\dfrac{A\tau_1+B}{C\tau_1+D},\\
& \text{there exists } \pt{a_{1},\ldots, a_{m+n}, b_{1}, \ldots, b_{m+n}} \in \Z^{2m+2n} \text{ with }\\ &\pt{a_{\ell+1}+b_{\ell+1} \r, \ldots, a_{m+n}+ b_{m+n} \r}\neq \mathbf{0}, \max \abs{a_i}, \abs{b_i} \leq T \text{ and }\\
& \sum\limits_{i=1}^{\ell}\!(a_i+b_i \r) \widetilde{z}_i + \sum\limits_{i=\ell+1}^{m}\!(a_i+b_i \r) z_i+(C\tau_1+D)\sum\limits_{j=1}^{n}(a_{m+j}+b_{m+j}\r)w_j \in \Z+ \Z\tau_1 \Bigg\rbrace
\end{align*}
where $\mathcal{Z}$ is the set defined in \eqref{def:Z_isog}, $\widetilde{z}_1,\ldots,\widetilde{z}_{\ell}$ are the elliptic logarithms of the constant points $\widetilde{P}_1, \ldots, \widetilde{P}_{\ell}$ and $\r$ is either $0$ in cases \ref{case_1} and \ref{case_2}, or a fixed quadratic integer in case \ref{case_3}.

The goal of this section is to prove the following result.
\begin{proposition}\label{prop:main_estim_isog}
Under the hypotheses of Theorem \ref{thm:main_thm_isog}, for all $\eps>0$, we have $\# \mathcal{Z}(T)\ll_{\eps} T^{\eps}$, for all $T\geq 1$.
\end{proposition}
To prove this, we will apply Proposition \ref{prop:thm_HP_isog} to the definable set $W$ consisting of tuples of the form
\begin{align*}
(&\a_{1}, \ldots, \a_{m+n},\b_{1}, \ldots, \b_{m+n}, A,B,C,D, \xi_1, \xi_2,\\
 &\zeta_1, \theta_1, \widetilde{x}_1, \widetilde{y}_1, \ldots,\widetilde{x}_{\ell}, \widetilde{y}_{\ell}, x_{\ell+1}, y_{\ell+1}, \ldots, x_m, y_m,\zeta_2, \theta_2, u_1, v_1, \ldots, u_n, v_n)
\end{align*}
in $\R^{2m+2n+6} \times \R^{2m+2n+4}$, satisfying the following relations:
$$\pt{\a_{\ell+1}+\b_{\ell+1} \r, \ldots,\a_{m+n}+ \b_{m+n} \r}\neq \mathbf{0} \qquad AD - BC \neq 0  \qquad C(\zeta_1 + \theta_1 \i) + D \neq 0$$
$$\pt{\zeta_1+\theta_1 \i,\widetilde{x}_1+\widetilde{y}_1 \i,\ldots, \widetilde{x}_{\ell}+\widetilde{y}_{\ell} \i, x_{\ell+1}+y_{\ell+1} \i, \ldots, x_m+y_m \i, \zeta_2+\theta_2 \i, u_1+v_1 \i, \ldots, u_n+v_n \i} \in \mathcal{Z}$$\vspace{-0.3cm}
$$\pt{C(\zeta_1+\theta_1 \i)+D}(\zeta_2+\theta_2 \i)=A(\zeta_1+\theta_1 \i)+B$$
\begin{align*}
\sum\limits_{p=1}^{\ell}(\a_p+\b_p \r)(\widetilde{x}_p+\widetilde{y}_p \i) &+ \sum\limits_{q=\ell+1}^{m}(\a_q+\b_q \r)(x_q+y_q \i)\\
&+\pt{C(\zeta_1+\theta_1 \i)+D}\sum\limits_{r=1}^{n}(\a_{m+r}+\b_{m+r}\r)(u_r+v_r \i)
=\xi_1+\xi_2(\zeta_1+\theta_1 \i)
\end{align*}
where $\i$ is the imaginary unit. In particular, we consider for each $ T\geq 1 $
\[ W^{\sim}(\Q,T):=\pg{(\a_{1}, \ldots, v_n)\in W: H_1(\a_{1}, \ldots, \a_{m+n},\b_{1}, \ldots, \b_{m+n}, A,B,C,D, \xi_1, \xi_2)\leq T} \]
where we recall that $H_1(\a_{1}, \ldots, \xi_2)$ is finite if and only if $(\a_{1}, \ldots, \xi_2) \in \Q^{2m+2n+6}$. 

Let $\pi_1, \pi_2$ be the projections on the first $2m+2n+6$ and the last $2m+2n+4$ coordinates, respectively.

\begin{lemma}\label{lemma:bound_ll_Teps_isog}
For every $\eps>0$, $\# \pi_2\pt{W^{\sim}(\Q,T)}\ll_{\eps} T^{\eps}$, for all $T\geq 1$.
\end{lemma}
\begin{proof}
Fix $\eps > 0$ and let $c=c(W,\eps)$ be the constant given by Proposition \ref{prop:thm_HP_isog}. Suppose also $\# \pi_2\pt{W^{\sim}(\Q,T)}>c T^{\eps}$ for some $T\geq 1$. 

Then, by Proposition \ref{prop:thm_HP_isog}, there exists a continuous definable function $ \delta: \left[0,1 \right] \rightarrow W$ such that its restriction to $(0,1)$ is real analytic, $\d_1=\pi_1 \circ \d:\pq{0,1} \rightarrow \R^{2m+2n+6}$ is semi-algebraic and $\d_2=\pi_2 \circ \d:\pq{0,1} \rightarrow \R^{2m+2n+4}$ is non-constant. Thus, there exists an infinite connected $J \subseteq \pq{0,1}$ such that $\d_1(J)$ is contained in an algebraic curve and $\d_2(J)$ has positive dimension.

Consider the coordinates 
\begin{align*}
&\a_{1}, \ldots, \a_{m+n},\b_{1}, \ldots, \b_{m+n}, A,B,C,D, \xi_1, \xi_2,\\
 &\zeta_1, \theta_1, \widetilde{x}_1, \widetilde{y}_1, \ldots,\widetilde{x}_{\ell}, \widetilde{y}_{\ell}, x_{\ell+1}, y_{\ell+1}, \ldots, x_m, y_m,\zeta_2, \theta_2, u_1, v_1, \ldots, u_n, v_n
\end{align*}
as functions on $J$ and define
$$\tau_{i}=\zeta_{i}+ \theta_{i} \i, \qquad \widetilde{z}_{p}=\widetilde{x}_{p}+\widetilde{y}_{p} \i, \qquad z_{q}=x_{q}+y_{q} \i \qquad w_{r}=u_{r}+v_{r} \i$$
with $i=1,2$, $p=1, \ldots, \ell$, $q=\ell+1, \ldots, m$ and $r=1, \ldots n$.

On $J$, the functions $\a_{1}, \ldots, \a_{m+n},\b_{1}, \ldots, \b_{m+n}, A,B,C,D, \xi_1, \xi_2$ satisfy $2m+2n+6-1=2m+2n+5$ independent algebraic relations over $\C$ (because they are functions on an algebraic curve). Since $\pt{\a_{\ell+1} + \b_{\ell+1} \r, \ldots, \a_{m+n} + \b_{m+n} \r}\neq \mathbf{0}$ and by the relations 
$$\sum\limits_{p=1}^{\ell}(\a_p+\b_p \r)\widetilde{z}_p + \sum\limits_{q=\ell+1}^{m}(\a_q+\b_q \r)z_q+\pt{C\tau_1+D}\sum\limits_{r=1}^{n}(\a_{m+r}+\b_{m+r}\r)w_r=\xi_1+\xi_2\tau_1 $$
$$\pt{C\tau_1+D}\tau_2=A\tau_1+B$$ 
it follows that the $2m+2n+6+(m-\ell)+n=3m+3n+6- \ell$ functions 
$$\a_{1}, \ldots, \a_{m+n},\b_{1}, \ldots, \b_{m+n}, A,B,C,D, \xi_1, \xi_2, z_{\ell+1}, \ldots, z_m, w_1, \ldots, w_n$$ satisfy $2m+2n+5+2=2m+2n+7$ independent algebraic relations over $F=\C\pt{\tau_1, \tau_2}$. Finally, let
$$\mathcal{W}=(\tau_1(J), \widetilde{z}_1,\ldots,\widetilde{z}_{\ell}, z_{\ell+1}(J), \ldots , z_m(J), \tau_2(J), w_1(J), \ldots, w_n(J)) \subseteq \mathcal{Z},$$
which has positive dimension since $\d_2(J)$ has positive dimension, and consider $\tau_1, z_{\ell+1},$ $\ldots , z_m,$ $\tau_2, w_1, \ldots, w_n$ as holomorphic functions on $\f(\mathcal{W}) \subseteq \c(\C)$. The algebraic relations found above can be analytically continued to an open disc $D$ in $\f(\mathcal{W}) \cap \widehat{\c}$. Therefore, 
$$\mathrm{trdeg}_F F\pt{z_{\ell+1}, \ldots, z_m, w_1, \ldots, w_n} \leq 3m+3n+6 - \ell-(2m+2n+7)=m+n-\ell-1$$
which implies that $z_{\ell+1}, \ldots, z_m, w_1, \ldots, w_n$ are algebraically dependent over $F$ on $D$ and thus, by Lemma \ref{lemma:func_transc_isog}, we get a contradiction, proving the proposition.
\end{proof}

\begin{proof}[Proof of Proposition \ref{prop:main_estim_isog}]
If $(\tau_1,\widetilde{z}_1,\ldots,\widetilde{z}_{\ell}, z_{\ell+1}, \ldots, z_m, \tau_2, w_1, \ldots, w_n) \in \mathcal{Z}(T)$, then there are integers $ a_1, \ldots, a_{m+n},b_1, \ldots, b_{m+n}$, $A,B,C,D $ with absolute value bounded by $T$ and integers $\xi_1, \xi_2$ such that 
\[ \pt{\a_{\ell+1}+\b_{\ell+1} \r, \ldots,\a_{m+n}+ \b_{m+n} \r}\neq \mathbf{0} \qquad AD - BC \neq 0  \qquad C\tau_1 + D \neq 0 \]
\[ \pt{C\tau_1+D}\tau_2=A\tau_1+B \]
\[ \sum\limits_{i=1}^{\ell}(a_i+b_i \r)\widetilde{z}_i + \sum\limits_{i=\ell+1}^{m}(a_i+b_i \r)z_i+\pt{C\tau_1+D}\sum\limits_{j=1}^{n}(a_{m+j}+b_{m+j}\r)w_j=\xi_1+\xi_2 \tau_1\]
And since $ \lvert \tau_1 \rvert, \lvert \tau_2 \rvert, \lvert A \rvert, \lvert B \rvert, \lvert C \rvert, \lvert D \rvert, \lvert a_{1} \rvert, \ldots,  \lvert a_{m+n} \rvert, \abs{b_{1}}, \ldots, \abs{b_{m+n}}\leq T$ and $\widetilde{z}_p, z_{q} \in \mathcal{L}_{\tau_1}$, $w_r \in \mathcal{L}_{\tau_2}$ we have that
\begin{align*}
&\left| \sum\limits_{p=1}^{\ell}(a_p+b_p \r)\widetilde{z}_p + \sum\limits_{q=\ell+1}^{m} (a_q +b_q \r)z_q + (C\tau_1+D)\sum\limits_{r=1}^{n} (a_{m+r}+b_{m+r}\r) w_r \right|\\
 &\leq \sum\limits_{p=1}^{\ell} \pt{\lvert a_p \rvert+ \abs{b_p} \abs{\r}}  \lvert \widetilde{z}_p \rvert + \sum\limits_{q=\ell+1}^{m} \pt{\lvert a_q \rvert+ \abs{b_q} \abs{\r}}  \lvert z_q \rvert + \lvert C\tau_1+D \rvert \sum\limits_{r=1}^{n} \pt{\lvert a_{m+r} \rvert +\abs{b_{m+r}}\abs{\r}} \lvert w_r \rvert \\
&\ll T \cdot \max \pg{1, \lvert \tau_1 \rvert}+ (T\cdot \max\pg{1, \abs{\tau_1}}) \cdot T \cdot \max \pg{1, \lvert \tau_2 \rvert} \ll T^4
\end{align*}
Therefore, we have 
$$\abs{\xi_1+ \xi_2\tau_1}=\abs{\sum\limits_{p=1}^{\ell}(a_p+b_p \r)\widetilde{z}_p + \sum\limits_{q=\ell+1}^{m} (a_q +b_q \r)z_q + (C\tau_1+D)\sum\limits_{r=1}^{n} (a_{m+r}+b_{m+r}\r) w_r}\ll T^4$$
from which we deduce that $\abs{\xi_2}\ll T^5$, since $\Im(\tau_1)\geq \frac{1}{T}$ and
$$T^4 \gg \abs{\xi_1+ \xi_2\tau_1} \geq \abs{\Im(\xi_1+ \xi_2\tau_1)}=\abs{\xi_2 \Im(\tau_1)}\geq \frac{\abs{\xi_2}}{T}.$$
It follows that $$\abs{\xi_1}=\abs{\xi_1+\xi_2\tau_1-\xi_2\tau_1}\leq \abs{\xi_1+\xi_2\tau_1}+\abs{\xi_2}\cdot \abs{\tau_1} \ll T^4+ T^5 \cdot T \ll T^6.$$
Hence we get that 
\begin{align*}
& \left( a_1, \ldots, a_{m+n},b_1, \ldots, b_{m+n}, A,B,C,D, \xi_1, \xi_2,\right. \\
& \left. \Re(\tau_1), \Im(\tau_1), \Re(\widetilde{z}_1), \Im(\widetilde{z}_1), \ldots, \Re(\widetilde{z}_{\ell}), \Im(\widetilde{z}_{\ell}), \Re(z_{\ell+1}), \Im(z_{\ell+1}), \ldots, \Re(z_m), \Im(z_m), \right.\\  
& \left. \Re(\tau_2), \Im(\tau_2), \Re(w_1), \Im(w_1), \ldots, \Re(w_n), \Im(w_n)\right) \in W^{\sim} (\Q, \nu T^6)
\end{align*} 
for some positive constant $\nu$.  
Finally, consider the map 
$$\setlength{\arraycolsep}{0pt}
\renewcommand{\arraystretch}{1.2}
  \begin{array}{ r c l }
    \mathcal{Z}(T) & {} \longrightarrow {} & \pi_2\pt{W^{\sim} (\Q, \nu T^6)}\\
     (\tau_1, \ldots, w_n)      & {} \longmapsto {} & \pt{\Re(\tau_1), \Im(\tau_1), \ldots, \Re(w_n), \Im(w_n)}.
  \end{array}$$
Since this map is injective, the conclusion follows from Lemma~\ref{lemma:bound_ll_Teps_isog}.
\end{proof}

\section{Arithmetic bounds}
Let $\c$ be as in Theorem \ref{thm:main_thm_isog} and let $\c'$ be the set of points $\cc \in \widehat{\c}(\C)$ such that there exists an isogeny $\phi_{\cc}:E_{\m(\cc)}\rightarrow E_{\l(\cc)}$ and $\mathbf{a}, \mathbf{b} \in \Z^{m+n}$ with $(a_{\ell+1}+b_{\ell+1} \r, \ldots, a_{m+n}+ b_{m+n}\r)\neq \mathbf{0}$ and
$$\sum\limits_{p=1}^{\ell} (a_p +b_p \r)\widetilde{P}_p + \sum\limits_{q=\ell+1}^{m} (a_q +b_q \r)P_q(\cc) + \sum\limits_{r=1}^{n} (a_{m+r}+b_{m+r}\r) \phi_{\cc}(Q_r(\cc))=O$$
where $\r$ is 0 in cases \ref{case_1} and \ref{case_2}, and a fixed generator for $\en(E_{\l_0})$ in case \ref{case_3}. Moreover, we can also assume that $\phi_{\cc}$ is a cyclic isogeny.

Since $\c$ is defined over $\Qal$, the curve $\widetilde{\c}=(J \circ \pi)(\c)$ is also defined over $\Qal$ and thus, for every $\cc \in \c'$, 
$$\pt{J(\l(\cc)),J(\m(\cc))} \in \widetilde{\c} \cap \bigcup_{N\geq 1} Y_0(N).$$
As all the modular curves $Y_0(N)$ are defined over $\Q$, all the points $\pt{J(\l(\cc)),J(\m(\cc))}$ are algebraic, which implies that also $\l(\cc)$ and $\m(\cc)$ are algebraic for every $\cc \in \c'$. From this, it follows that $\c'$ is a subset of $\c(\Qal)$ and thus we can define, for every $\cc_0 \in \c'$, $D_0:=\pq{k(\l(\cc_0),\m(\cc_0)):k}$, where $k$ is the field of definition of $\c$. 

All the constants appearing in this section depend only on $\c$, the field $k$ and on the integers $m, n$, unless otherwise stated.

\begin{lemma}\label{lemma_bounds}
Let $\cc_0 \in \c'$ and let $N_0$ be the minimal degree of an isogeny $\phi_{\cc_0}: E_{\m(\cc_0)} \rightarrow E_{\l(\cc_0)}$. Then, for every $\eps>0$, there exist positive constants $\g_1, \g_2$ (depending on $\eps$) such that
$$h(\l(\cc_0)),\; h(\m(\cc_0)) \leq \g_1 D_0^{\eps} \quad \text{ and } \quad N_0 \leq \g_2 D_0^{2+\eps}.$$
\end{lemma}
\begin{proof}
Fix $\cc_0 \in \c'$, and let $N_0$ be as above. Note that an isogeny of minimal degree between two elliptic curves is necessarily cyclic, since otherwise it would factor as the composition of a cyclic isogeny and a multiplication-by-$n$ map.

Therefore, $(J(\l(\cc_0)), J(\m(\cc_0))) \in (\widetilde{\c} \cap Y_0(N_0))(\Qal)$. 
Since $\widetilde{\c}$ is asymmetric, we may apply \cite[Theorem~1.1]{Hab10} and standard properties of heights to deduce
$$h(\l(\cc_0)),\; h(\m(\cc_0)) \ll \log(1+N_0).$$
Next, by Théorème 1.4 of \cite{GR14}, we have
$$N_0 \ll D_0^2 \cdot \max\pg{h_{F}(E_{\l(\cc_0)}), \log(D_0),1}^2,$$
where $h_{F}(E_{\l(\cc_0)})$ is the (stable) Faltings height of $E_{\l(\cc_0)}$. By Proposition 2.1 of \cite{Sil86}, one has
$$h_F(E_{\l(\cc_0)}) \ll h\pt{j(E_{\l(\cc_0)})}+1.$$
Since $j(E_{\l(\cc_0)})$ is a rational function in $\l(\cc_0)$, Proposition 3.2 of \cite{Zan14} gives
$$h\pt{j(E_{\l(\cc_0)})}+1 \ll h(\l(\cc_0))+1 \ll \log(1+N_0) \ll_{\epsilon_1} N_0^{\epsilon_1}$$
for every $\epsilon_1 >0$.
Moreover, for every $\epsilon_2 >0$ we have $\log D_0 \ll_{\epsilon_2} D_0^{\epsilon_2}$. Hence
$$N_0 \ll D_0^2 \cdot \max\pg{h_{F}(E_{\l(\cc_0)}), \log(D_0),1}^2\ll_{\epsilon_1, \epsilon_2} D_0^{2+2\epsilon_2} \cdot N_0^{2\epsilon_1}.$$
Now fix $\eps>0$. Choosing $\epsilon_1 = \frac{\eps}{8+4\eps}$ and $\epsilon_2 = \frac{\eps}{4}$ yields
$$N_0 \ll_{\eps} D_0^{2+\eps}.$$
Finally, recalling that $h(\l(\cc_0)), h(\m(\cc_0)) \ll \log(1+N_0) \ll_\epsilon N_0^\epsilon$, we conclude
$$h(\l(\cc_0)),\; h(\m(\cc_0)) \ll_\eps D_0^\eps$$
for every $\eps>0$.
\end{proof}

\begin{remark}\label{rmk_asym}
Note that this lemma above is the only part of the proof where we need to use the hypothesis that $\c$ is asymmetric, while all the other steps are true also for non-asymmetric curves. Thus, if one was able to prove this lemma for an arbitrary $\c$ or any of the Conjectures 21.20, 21.23 or 21.24 from \cite{Pil22}, then Theorem \ref{thm:main_thm_isog} would follow for any $\c$.
\end{remark}

\begin{lemma}
Let $\cc_0 \in \c'$. Then, there exist positive constants $\g_3, \g_4, \g_5$ such that
$$\widehat{h}(P_q(\cc_0))\leq \g_3 D_0 \quad \text{for every } q={\ell+1}, \ldots, m$$
$$\widehat{h}(\phi_{\cc_0}(Q_r(\cc_0)))\leq \g_4 D_0^{4} \quad \text{for every } r=1, \ldots, n.$$
Moreover, the $P_q(\cc_0)$ and the $\phi_{\cc_0}(Q_r(\cc_0))$ are defined over a field $K\supseteq k(\l(\cc_0),\m(\cc_0))$ with $\pq{K:\Q}\leq \g_5 D_0^{2}$.
\end{lemma}
\begin{proof}
We use the same notation as in the previous proof.

Using work of Zimmer \cite[Theorem]{Zim76}, the previous lemma (with $\eps=1$) and the bounds (\ref{eqn:height_bounds_P_lambda}), in case \ref{case_1} we have 
$$\widehat{h}(P_p(\cc_0))\leq h(P_p(\cc_0))+\g_6\pt{h(\l(\cc_0))+1}\leq \g_7\pt{h(\l(\cc_0))+1} \leq \g_3 D_0$$
while in cases \ref{case_2} and \ref{case_3} we get the same estimate by
$$\widehat{h}(P_p(\cc_0))\leq h(P_p(\cc_0))+\g_6\pt{h(\m(\cc_0))+1}\leq \g_7\pt{h(\m(\cc_0))+1} \leq \g_3 D_0.$$
Similarly, $\widehat{h}(Q_r(\cc_0))\leq \g_8 D_0$. So, by Proposition \ref{PropHIso} and Lemma \ref{lemma_bounds} (again with $\eps=1$), we get that 
$$\widehat{h}(\phi_{\cc_0}(Q_r(\cc_0)))\leq \g_8 N_0 D_0 \leq \g_4 D_0^{4}$$
Lemma 7.2 in \cite{BC16} implies that the $P_p(\cc_0)$ and the $Q_r(\cc_0)$ are defined over a field $K_1$ of degree at most $\g_{9} D_0$ over $\Q$. Moreover, by \cite[Lemma 6.1]{MW90}, $\phi_{\cc_0}$ is defined over a field $K_2$ of degree at most $12$ over $k(\l(\cc_0),\m(\cc_0))$, and thus $[K_2: \Q] \leq 12 D_0$. 
Therefore, the points $\phi_{\cc_0}(Q_r(\cc_0))$ are defined over the compositum $K_1 K_2$ which has degree $\ll D_0^{2}$ over $\Q$.
\end{proof}

Next, we show that for any $\cc_0 \in \c'$ we can choose \quotes{small} coefficients for the linear relation. 

\begin{lemma}\label{lemma_coeff}
For any $\cc_0 \in \c'$, there exist $\mathbf{a}, \mathbf{b} \in \Z^{m+n}$ with $(a_{\ell+1}+b_{\ell+1} \r, \ldots, a_{m+n}+ b_{m+n}\r)\neq \mathbf{0}$ and
$$\sum\limits_{p=1}^{\ell} (a_p +b_p \r)\widetilde{P}_p + \sum\limits_{q=\ell+1}^{m} (a_q +b_q \r)P_q(\cc_0) + \sum\limits_{r=1}^{n} (a_{m+r}+b_{m+r}\r) \phi_{\cc_0}(Q_r(\cc_0))=O$$
and such that $\max \pg{\abs{a_i}, \abs{b_i}} \leq \g_{10} D_0^{\eta_1}$ for some positive constants $\g_{10}, \eta_1$.
\end{lemma}
\begin{proof}
For cases \ref{case_1} and \ref{case_2}, we already saw that we can take $\r=0$ and therefore we can choose $\mathbf{b}=\mathbf{0}$. So we can apply Lemma 6.1 of \cite{BC16} (which is in turn based on a result by Masser \cite{Mas88}), to the points $\widetilde{P}_p, P_q(\cc_0)$ and $\phi_{\cc_0}(Q_r(\cc_0))$, using also Lemma \ref{lemma_bounds} and the height bounds from the previous lemma.

In case \ref{case_3}, we use again the above-mentioned lemma by Barroero and Capuano, this time with the points $\widetilde{P}_p, \r \widetilde{P}_p, P_q(\cc_0)$, $\r P_q(\cc_0)$, $\phi_{\cc_0}(Q_r(\cc_0))$ and $\r \phi_{\cc_0}(Q_r(\cc_0))$, recalling that by \eqref{diseq_alt_rhoP}, we have that
$$\widehat{h}(\r P_q(\cc_0))\ll D_0^6 \cdot \widehat{h}(P_q(\cc_0))\ll D_0^{7}$$
$$\widehat{h}(\r \phi_{\cc_0}(Q_r(\cc_0)))\ll D_0^6 \cdot \widehat{h}(\phi_{\cc_0}(Q_r(\cc_0)))\ll D_0^{10}.$$
A priori, applying \cite[Lemma 6.1]{BC16} in cases \ref{case_2} and \ref{case_3} gives $\mathbf{a}, \mathbf{b} \in \Z^{m+n}$ with $(a_{1}+b_{1} \r, \ldots, a_{m+n}+ b_{m+n}\r)\neq \mathbf{0}$, but we claim that we cannot have $(a_{\ell+1}+b_{\ell+1} \r, \ldots, a_{m+n}+ b_{m+n}\r)= \mathbf{0}$, otherwise this would mean that the constant points $\widetilde{P}_1, \ldots, \widetilde{P}_{\ell}$ are $\en(E_{\l_0})$-linearly dependent, contradicting our assumptions.
\end{proof}

For the next lemma, let $\tau_1(\cc)=\tau_1(\f^{-1}_{\vert_{\mathcal{F}_{\mathcal{B}}}}(\cc)) \in \mathcal{B}$ for every $\cc \in \c(\C)$ and similarly for $\tau_2(\cc)$, where $\f$ is the uniformization map defined in Section \ref{sec_unif}.

\begin{lemma}\label{lemma_tau}
There exist positive constants $\g_{11}, \g_{12}$ such that for every $\cc_0 \in \c'$ we have
$$\lvert \tau_1(\cc_0) \rvert, \lvert \tau_2(\cc_0) \rvert \leq \g_{11} D_0^{2} \quad \text{ and } \quad \Im(\tau_1(\cc_0)), \Im(\tau_2(\cc_0)) \geq \g_{12} \dfrac{1}{D_0^4}.$$
\end{lemma}
\begin{proof}
Let $\mathfrak{F}$ be the usual fundamental domain for the action of $\mathrm{SL}_2(\Z)$ on $\mathbb{H}$ (defined in (\ref{eqn:def_fund_dom})) and let $\tau \in \mathfrak{F}$. Then, Lemma 1 in \cite{BMZ13} implies that $e^{2 \pi \Im(\tau)} \leq 2079+\abs{j(\tau)}$. Hence, if $\abs{j(\tau)}\leq 2$, then $\Im(\tau)\leq \frac{1}{2\pi}\log(2081)=\g_{13}$. Equivalently, for every $\tau \in \mathfrak{F}$ such that $\Im(\tau)>\g_{13}$, we have $\abs{j(\tau)}>2$.
So, if $\Im(\tau)> \g_{13}$, we then get that 
$$\Im(\tau)\leq \dfrac{1}{2\pi}\log\pt{2079+\abs{j(\tau)}}\leq \dfrac{\log(2081)}{2\pi \log(2)}\log\abs{j(\tau)}.$$
Therefore, for every $\tau \in \mathfrak{F}$ we have $\Im(\tau) \ll \max\pg{1, \log\abs{j(\tau)}}$.

Now, assume that $\tau_1(\cc_0)=M \cdot \tau_1'$, for some $\tau_1' \in \mathfrak{F}$ and $M=\begin{psmallmatrix}
a & b \\ 
c & d
\end{psmallmatrix}\in \mathrm{SL}_2(\Z)$. Then,
$$\Im(\tau_1(\cc_0))=\Im\!\pt{\dfrac{a \tau_1'+b}{c \tau_1'+d}}=\dfrac{\Im(\tau_1')}{\abs{c\tau_1'+d}^2}\leq \dfrac{\Im(\tau_1')}{c^2-cd+d^2}\ll_{\,\!{}_{M}} \max\pg{1, \log\abs{j(\tau_1')}}.$$
As the $j$-function is invariant under the action of $\mathrm{SL}_2(\Z)$, we have that 
$$j(\tau_1(\cc_0))=j(M \cdot \tau_1') =j(\tau_1'),$$
so that $\Im(\tau_1(\cc_0))\ll_{\,\!{}_{M}} \max\pg{1, \log\abs{j(\tau_1(\cc_0))}}$.

Furthermore, we have that $j(\tau_1(\cc_0))=J(L(\tau_1(\cc_0)))=J(\l(\cc_0))$, where $J(\l)=2^8 \frac{(\l^2-\l+1)^3}{\l^2 (\l-1)^2}$ and $L$ was defined in (\ref{def_L}). Since $\l(\cc_0) \in \Qal\setminus\pg{0,1}$, this implies that $j(\tau_1(\cc_0)) \in \Qal$.

Then, using the inequality $\log \abs{\a} \leq \pq{\Q(\a):\Q} h(\a)$ for every non-zero $\a \in \Qal$, we get
\begin{align*}
\log\abs{j(\tau_1(\cc_0))}&\leq \pq{\Q(j(\tau_1(\cc_0))):\Q} h(j(\tau_1(\cc_0)))=\pq{\Q(J(\l(\cc_0))):\Q} h(J(\l(\cc_0)))\\
&\ll \pq{\Q(\l(\cc_0)):\Q} (h(\l(\cc_0))+1) \ll D_0^{2} 
\end{align*}
by Lemma \ref{lemma_bounds} and \cite[Proposition 3.2]{Zan14}.
Combining this with the previous bound gives $\Im(\tau_1(\cc_0))\leq \g_{14}(M) D_0^{2}$, for some positive constant $\g_{14}(M)$ depending on $M$. Since $\mathcal{B}$ is the union of finitely many translates of $\mathfrak{F}$ by elements of $\mathrm{SL}_2(\Z)$, there are only finitely many such $M$ to consider. Thus, we have that $\Im(\tau_1(\cc_0))\leq \g_{15} D_0^{2}$, where $\g_{15}$ is an absolute constant, and that $\abs{\Re(\tau_1(\cc_0))} \ll 1$, as $\mathcal{B}$ is a set of bounded real part. Therefore, we get that $\abs{\tau_1(\cc_0)} \ll D_0^2$.

For the lower bound on the imaginary part, first note that if $\tau \in \mathfrak{F}$, then $\Im(\tau)\geq \frac{\sqrt{3}}{2}$. Again, assume that $\tau_1(\cc_0)=M \cdot \tau_1'$, for some $\tau_1' \in \mathfrak{F}$ and $M=\begin{psmallmatrix}
a & b \\ 
c & d
\end{psmallmatrix}\in \mathrm{SL}_2(\Z)$. Then,
$$\Im(\tau_1(\cc_0))=\Im\pt{\dfrac{a \tau_1'+b}{c \tau_1'+d}}=\dfrac{\Im(\tau_1')}{\abs{c\tau_1'+d}^2}\geq \dfrac{\Im(\tau_1')}{(\abs{c}+\abs{d})^2 \cdot \max\pg{1, \abs{\tau_1'}}^2}\gg_{\,\!{}_{M}} \dfrac{1}{\max\pg{1, \abs{\tau_1'}}^2}.$$
From before we get that $\Im(\tau_1')\ll \max\pg{1, \log\abs{j(\tau_1')}}= \max\pg{1, \log\abs{j(\tau_1(\cc_0))}} \ll D_0^2$, which implies $\abs{\tau_1'}\ll D_0^2$, since $\tau_1' \in \mathfrak{F}$ implies $\abs{\Re(\tau_1')}\leq \frac{1}{2}$. Hence,
$$\Im(\tau_1(\cc_0))\gg_{\,\!{}_{M}} \dfrac{1}{\max\pg{1, \abs{\tau_1'}}^2}\gg \dfrac{1}{D_0^4}.$$
As before, we need to consider only finitely many choices of $M \in \mathrm{SL}_2(\Z)$, so we have that $\Im(\tau_1(\cc_0))\gg 1/{D_0^4}$, where the implied constant is absolute.

Similar arguments give the respective bounds for $\tau_2(\cc_0)$.
\end{proof}

\section{Proof of Theorem \ref{thm:main_thm_isog}}
We want to show that the set $\c'$ (defined at the start of the previous section) is finite. Since the map 
$$\pi_{\vert_{\c}}: \cc \mapsto (\l(\cc), \m(\cc))$$
is finite-to-one, Northcott’s theorem together with Lemma \ref{lemma_bounds} reduces the problem to bounding the degree $D_0$ of $\l(\cc)$ and $\m(\cc)$ over $k$.

Let $\cc_0 \in \c'$ and $\s \in \Gal(\overline{k}/k)$. Notice that $\s(\cc_0) \in \c'$. Indeed, we have
$$j\pt{E_{\l(\s(\cc_0))}}=j\pt{E_{\s(\l(\cc_0))}}=J\pt{\s(\l(\cc_0))}=\s\pt{J\pt{\l(\cc_0)}}=\s\pt{j\pt{E_{\l(\cc_0)}}}$$
and an analogous identity holds for $\m(\cc_0)$ in place of $\l(\cc_0)$. If $N_0$ is defined as in Lemma \ref{lemma_bounds}, it follows that 
$$\Phi_{N_0}\pt{j\pt{E_{\l(\s(\cc_0))}},j\pt{E_{\m(\s(\cc_0))}}}=\s\pt{\Phi_{N_0}\pt{j\pt{E_{\l(\cc_0)}},j\pt{E_{\m(\cc_0)}}}}=0,$$
so there exists a cyclic isogeny $\phi_{\s(\cc_0)}: E_{\m(\s(\cc_0))} \to E_{\l(\s(\cc_0))}$ of degree $N_0$. Since $\deg \s(\phi_{\cc_0})= \deg \phi_{\cc_0}=N_0$, we can take $\phi_{\s(\cc_0)}=\s\pt{\phi_{\cc_0}}$.

Thus, as $\c$ is defined over $k$, we also have
$$\s(\widetilde{P}_p)=\widetilde{P}_p \, , \qquad P_q(\s(\cc_0))=\s\pt{P_q(\cc_0)},$$ 
$$\phi_{\s(\cc_0)}\pt{Q_r(\s(\cc_0))}=\phi_{\s(\cc_0)}\pt{\s\pt{Q_r(\cc_0)}}=\s \pt{\phi_{\cc_0}\pt{Q_r(\cc_0)}}$$
for every $p=1, \ldots, \ell$, $q=\ell+1, \ldots, m$, $r=1, \ldots, n$.
Moreover, in case \ref{case_3}, we can assume without loss of generality that the generator $\r$ of $\en(E_{\l_0})$ is defined over $k$, so that $\s(\r)=\r$. Recall that we are using \cite[Proposition II.1.1]{Sil94} to identify endomorphisms with complex numbers. Furthermore, \cite[Proposition II.2.2]{Sil94} guarantees that under this identification the action of $\Gal(\overline{k}/k)$ on these two objects is the same.

So, in all cases, we have:
\begin{align*}
&\sum\limits_{p=1}^{\ell} (a_p +b_p \r)\widetilde{P}_p +\sum\limits_{q=\ell+1}^{m} (a_q +b_q \r)P_q(\s(\cc_0)) + \sum\limits_{r=1}^{n} (a_{m+r}+b_{m+r}\r) \phi_{\s(\cc_0)}(Q_q(\r(\cc_0)))\\
&=\s\Big(\sum\limits_{p=1}^{\ell} (a_p +b_p \r)\widetilde{P}_p +\sum\limits_{q=\ell+1}^{m} (a_q +b_q \r)P_q(\cc_0) + \sum\limits_{r=1}^{n} (a_{m+r}+b_{m+r}\r) \phi_{\cc_0}(Q_r(\cc_0))\Big)=O
\end{align*}
on $E_{\l(\s(\cc_0))}$, since the $a_i$ and $b_i$ are integers.

Now, consider the point $\f^{-1}_{\vert_{\mathcal{F}_{\mathcal{B}}}}\pt{\s(\cc_0)} \in \mathcal{Z}$ with coordinates $$\pt{\tau_1^{\s},\widetilde{z}_1, \ldots, \widetilde{z}_{\ell}, z_{\ell+1}^{\s},\ldots, z_m^{\s}, \tau_2^{\s}, w_1^{\s}, \ldots, w_n^{\s}}$$
(here the superscript $\s$ does not denote a Galois conjugate). By the previous equation and lemmas \ref{lemma_coeff} and \ref{lemma_tau} we have relations
$$\sum\limits_{p=1}^{\ell} (a_p +b_p \r)\widetilde{z}_p +\sum\limits_{q=\ell+1}^{m} (a_q +b_q \r)z_q^{\s} + (C^{\s}\tau_1^{\s}+D^{\s})\pt{\sum\limits_{r=1}^{n} (a_{m+r}+b_{m+r}\r) w_r^{\s}} \in \Z+\Z\tau_1^{\s},$$
$$\tau_2^{\s}=\dfrac{A^{\s}\tau_1^{\s}+B^{\s}}{C^{\s}\tau_1^{\s}+D^{\s}}$$
with $(a_{\ell+1}+b_{\ell+1} \r, \ldots, a_{m+n}+ b_{m+n}\r)\neq \mathbf{0}$ and
$$\max \pg{\abs{a_i}, \abs{b_i}} \leq \g_{10} D_0^{\eta_1}, \qquad \abs{\tau_1^{\s}},\abs{\tau_2^{\s}}\leq \g_{11} D_0^{2}, \qquad \Im(\tau_1^{\s})\geq \g_{12} \dfrac{1}{D_0^4} $$
and $\abs{A^{\s}},\abs{B^{\s}},\abs{C^{\s}},\abs{D^{\s}}\leq \g_{16}N_0\leq \g_{17} D_0^{3}$ by Lemma \ref{lemma_bound_isogenie_matrix} and Lemma \ref{lemma_bounds}.

So, $$\f^{-1}_{\vert_{\mathcal{F}_{\mathcal{B}}}}\pt{\s(\cc_0)} \in \mathcal{Z}\pt{\g D_0^{\eta}}$$
where $\g=\max\pg{\g_{10},\g_{11}, \frac{1}{\g_{12}}, \g_{17}}$ and $\eta=\max\pg{\eta_1, 4}$.

Finally, there are at least $\pq{k(\cc_0):k}\geq \pq{k(\l(\cc_0),\m(\cc_0)):k}=D_0$ different 
$$\pt{\tau_1^{\s},\widetilde{z}_1, \ldots, \widetilde{z}_{\ell}, z_{\ell+1}^{\s},\ldots, z_m^{\s}, \tau_2^{\s}, w_1^{\s}, \ldots, w_n^{\s}}$$ in $\mathcal{Z}\pt{\g D_0^{\eta}}$. However, applying Proposition \ref{prop:main_estim_isog} with $\eps=\frac{1}{2\eta}$ gives a contradiction if $D_0$ is large enough. This proves that $D_0$ is bounded and, consequently, Theorem \ref{thm:main_thm_isog}.

\subsection*{Acknowledgements}
We would like to thank Fabrizio Barroero and Laura Capuano for many useful discussions and comments, and Gabriel Dill for his comments, for pointing out an alternative proof of cases (ii) and (iii), and for his hospitality in Bonn. We are also grateful to Francesco Veneziano and to the anonymous referee for their many helpful comments, which greatly improved the exposition.

The author was supported by the PRIN 2022 project \emph{2022HPSNCR: Semiabelian varieties, Galois representations and related Diophantine problems} and the \emph{National Group for Algebraic and Geometric Structures, and their Applications} (GNSAGA INdAM).

\bibliographystyle{alpha}
\bibliography{Isogeny_relations_biblio}

\end{document}